\documentclass{article}
\usepackage{amsmath,amsthm,amsfonts,graphicx}
\usepackage{multirow}

\DeclareMathOperator{\diag}{diag}

\DeclareMathOperator{\rank}{rank}

\DeclareMathOperator{\sign}{sign}

\DeclareMathOperator{\trace}{trace}

\def\smallddots{\mathinner{\raise7pt\hbox{.}\raise4pt\hbox{.}\raise1pt\hbox{.}}}
\def\smallsdots{\mathinner{\raise1pt\hbox{.}\raise4pt\hbox{.}\raise7pt\hbox{.}}}

\numberwithin{equation}{section}
\numberwithin{table}{section}

\newtheorem{theorem}{Theorem}[section]

\newtheorem{corollary}{Corollary}[section]
\newtheorem{fact}{Fact}[section]
\newtheorem{algorithm}{Algorithm}[section]

\newtheorem{definition}{Definition}[section]

\newtheorem{remark}{Remark}[section]

\begin{document}

%------------------------------------------------------------------------------

\title{\bf New Structured Matrix Methods for Real and Complex Polynomial Root-finding 
\thanks {Some results of this paper have been presented at the 
14th Annual Conference on Computer Algebra in Scientific Computing (CASC '2012), 
September 2012, Maribor, Slovenia, and at the 
18th Conference of the International Linear Algebra
Society (ILAS'2013), Providence, RI, June 2013}
%ACM SIGSAM
%International Symposiums on Symbolic and Algebraic
%Computation (ISSAC 2010) in M\"{u}nchen, Germany, in July 2010, the
%3nd International Conference on Matrix Methods in Mathematics and 
%Applications (MMMA 2011),  in
%Moscow, Russia, June 22-25, 2011 (plenary talk), and 
%the 7th International Congress on Industrial and Applied Mathematics 
%(ICIAM 2011), in Vancouver, British Columbia, Canada, July 18-22, 2011.
%}
}
\author{Victor Y. Pan$^{[1, 2],[a]}$
%, Guoliang Qian$^{[2],[b]}$, 
and Ai-Long Zheng$^{[2],[b]}$  \\
Supported by NSF Grant CCF-1116736 and PSC CUNY Award 64512--0042
\and\\
$^{[1]}$ Department of Mathematics and Computer Science \\
Lehman College of the City University of New York \\
Bronx, NY 10468 USA \\
$^{[2]}$ Ph.D. Programs in Mathematics  and Computer Science \\
The Graduate Center of the City University of New York \\
New York, NY 10036 USA \\
$^{[a]}$ victor.pan@lehman.cuny.edu \\
http://comet.lehman.cuny.edu/vpan/  \\
$^{[b]}$
% gqian@gc.cuny.edu $^{[c]}$
 azheng-1999@yahoo.com \\
}
 \date{}

%------------------------------------------------------------------------------

\maketitle

%------------------------------------------------------------------------------

\begin{abstract}
We combine the known methods for
univariate
polynomial  root-finding and
for computations 
in the Frobenius matrix algebra with 
our novel techniques to advance 
numerical 
solution of a univariate polynomial equation,  
and in particular 
numerical approximation of
the real
roots of a  polynomial.
Our analysis and experiments show efficiency of the resulting 
algorithms. 
%Our auxiliary results on
%randomized matrix computations
%can be of independent interest.
\end{abstract}

%------------------------------------------------------------------------------


%------------------------------------------------------------------------------


%------------------------------------------------------------------------------

\section{Introduction}\label{sintr}

%-----------------------------------------------------------------------------

Polynomial root-finding is the oldest subject of mathematics and computational
mathematics and is still an area of intensive research worldwide.
The list of 
   hundreds if not thousands algorithms  known
for this task
 still
 grows every year
(see
the books and articles \cite{B40},  \cite{B68}, \cite{C99}, \cite{P97},
\cite{P98},
 \cite{MN02}, \cite{MN07}, \cite{MNPb}, \cite{PT13},
%\cite{BGP02/04},
%\cite{BGP03/05},
%\cite{ESY06}, \cite{EMT08}, \cite{TE06}, \cite{HTZ09},
%\cite{BBEGG},
%\cite{VBV10},
%\cite{P11},
%\cite{PZ10/11}, \cite{PZ11},
%\cite{MS11}, \cite{YS11},  \cite{MNP12},
%\cite{GA12}, \cite{S12}, \cite{ST12}, \cite{SY12},
and the bibliography therein).
Many algorithms are directed to computing
a single, e.g., absolutely largest root 
 of a polynomial
 or a subset of all its $n$ roots,
e.g., all $r$ its real roots.
In some applications, e.g.,
to algebraic geometric
optimization, only the real roots are of interest,
and they can be much less numerous than all $n$
complex roots.
Nevertheless
the best numerical subroutines such as MPSolve
approximate all these $r$ real roots
about as fast and as slow as all $n$ complex roots.

Root-finding for a polynomial $p(x)$
via ei\-gen-solv\-ing for the associated companion matrix $C_p$
is a classical approach recently revived,
with the
incorporation of the well developed
numerical matrix methods
(see \cite{BDD00},  \cite{GL96}, \cite{S01}, \cite{W07},
and the bibliography therein).
The QR algorithm, adopted for polynomial root-finding by Matlab,
avoids numerical problems, faced by many other companion matrix methods
\cite[Section 7.4.6]{GL96}, but is not readily amenable to 
exploiting the rich structure of the companion matrix.
Extensive research toward such exploitation by using 
 QR- and LR-based root-finders has been initiated in the papers \cite{BGP03/05},
\cite{BGP04} and \cite{BDG04} and still goes on
(see \cite{BBEGG},
\cite{VBV10}, \cite{Z12}, \cite{AVW13}, 
%\cite{BBC11},
and the references therein). The  QR algorithm is celebrated 
for  its fast empirical convergence, but
the Rayleigh Quotient iteration \cite[Section 8.2.2]{GL96}
 also has very good convergence  record, 
exploits matrix structures even better
than  the QR algorithm,
and unlike that algorithm
can be  applied concurrently with no communication
among the processors that handle
 distinct initial points.
The papers  \cite{BGP02/04},  \cite{PZ10/11}
 adjust this iteration to polynomial root-finding  and
perform
every iteration step and every deflation step in linear space and linear 
arithmetic time.

In this paper we explore the somewhat similar approach of Cardinal \cite{C96},
extended in \cite{BP96} and \cite{P05}.
It enhances the Power Method and the method
of \cite{S41}, \cite{S70}, and \cite{H71}
by reducing every multiplication in the
Frobenius algebra, generated by the companion matrix $C_p$,
to application of a small number of FFTs.
By 
combining these
and some other known 
techniques of polynomial root-finding
 with our novelties,
we achieve substantial progress,
in particular for numerical approximation of the real roots. 
We 
reduce this task
 to the approximation of 
the associated eigenspace of the companion matrix
(cf. Theorem \ref{thsubs}), 
make this  eigenspace dominant by using shifts, inversions 
and repeated squaring in the Frobenius matrix algebra
as well as the approximation of the matrix sign function,
and then readily approximate this eigenspace and the associated
eigenvalues. Numerically we approximate 
the $r_+$
real and nearly roots of the input polynomial, and
among them we immediately select all
the $r$ real roots 
(see Remark \ref{renum} in Section \ref{ssse}).
In this way we accelerate the known numerical real root-finders by 
a factor of $n/r_+$ for a polynomial of  a degree $n$.
We
 also substantially accelerate the known numerical algorithms
for complex roots of polynomials
by proposing some novel matrix methods, as we show both
formally and empirically.

%to global convergence of Newton's classical iteration
%\begin{equation}\label{eqnewt}
%x^{(k+1)}=x^{(k)}-\frac{p(x^{(k)})}{p'(x^{(k)})},~k=0,1,\dots
%\end{equation}

\bigskip

%We also note a promising application
%to root-finding in combimation with the highly successful matrix method due to Malek and Vaillantcourt(1995)
%and Fortune (2002) implemented in Eigensolve, which is among the two leading packages of subrotines
%for polynomial root-finding. This method recursively alternates application of the QR and
%the Durand--Kerner (Weierstass) algorithms, and Kirrinnis' algorithm can become an
%effective alternative to Durand--Kerner's.

We organize our presentation as follows.
The next section is devoted to definitions and preliminary results.
In Section \ref{ssse} we present our basic algorithms.
They reduce the eigenvalue problem to the approximation 
of the dominant or dominated eigenspaces of the appropriate functions
of the input matrix. In the subsequent sections we cover
the computation of such matrix functions.
In Section \ref{srs} we do this
by combining repeated squaring, shifts and inversions 
in the associated matrix algebra, whereas in Section
\ref{smsf} we exploit the approximation of the matrix sign function.
Both sections are mostly
devoted to the approximation of real eigenvalues,
but  
 Subsections  \ref{srssr},
\ref{smsres} and
\ref{snum1}
 present some novel efficient algorithms  
that approximate complex  eigenvalues
of the companion matrix and consequently 
complex roots of a polynomial.
%In Section \ref{srs} we combine them with repeated
%squaring to approximate absolutely largest roots
%as well as the roots closest to a selected complex point.
%In Section \ref{smsf} we recall the matrix sign function.
%In Section \ref{smsres}
% we apply it to ei\-gen-solv\-ing.
%We cover its computation,
%adjust it to real ei\-gen-solv\-ing
%and modify it to save matrix inversions in Sections \ref{snumsgn}--\ref{scnrs}.
Section \ref{sexp} covers our numerical tests,
which are the contribution of the second author.
There are many directions for extending and refining our 
techniques, and
our concluding Section \ref{sconc} lists some of them.
In the Appendix we sketch a dual approach
emulating some of our techniques in terms of polynomial computations.
%which is a more promising basis for the extension to root-finding
%for a system of multivariate polynomials.

%------------------------------------------------------------------------------

\section{Definitions and preliminaries}\label{sdef}

%------------------------------------------------------------------------------

Hereafter ``flop" stands for ``arithmetic operation",
``is expected" and ``is likely" mean ``with a probability near 1",
and ``small", ``large", ``close", and ``near" are meant in the context.
We assume computations in the fields of complex and real numbers
$\mathbb C$ and $\mathbb R$,
respectively. For $\rho'>\rho>0$ and a complex $c$, define
the circle $\mathcal C_{\rho}(c)=\{\lambda:~|\lambda-c|=\rho\}$,
the disc
$\mathcal D_{\rho}(c)=\{\lambda:~|\lambda-c|\le \rho\}$,
and
the
annulus
$\mathcal A_{\rho,\rho'}(c)=\{\lambda:~\rho\le|\lambda-c|\le \rho'\}$.
A scalar $\lambda$ is {\em nearly real}
(within $\epsilon>0$)
if $|\Im(\lambda)|\le \epsilon |\lambda|$.

%------------------------------------------------------------------------------

\medskip

{\bf Matrix computations: fundamentals} \cite{GL96},  \cite{S98}, \cite{W02}.
$(B_j)_{j=1}^s=(B_1~|~B_2~|~\dots~|~B_s)$ is the $1\times s$  block matrix
with the blocks $B_1,B_2,\dots,B_s$.
$\diag(B_j)_{j=1}^s=\diag(B_1,B_2,\dots,B_s)$ is the $s\times s$  block diagonal matrix
with the diagonal blocks $B_1$, $B_2$,$\dots$, $B_s$.
$M^T$ is the transpose of a matrix $M$.
$\mathcal R(M)$ is the range of
a matrix $M$,
that is the linear space
generated by its columns. $\mathcal N(M)=\{{\bf v}:~M{\bf v}={\bf 0}\}$
is its null space.
$\rank (M)=\dim (\mathcal R(A))$.
%${\bf v}$ is its null vector if $M{\bf v}={\bf 0}$.
A matrix of full column rank is a {\em matrix basis} of its range.
$I=I_n=({\bf e}_1~|~{\bf e}_2~|\ldots~|~{\bf e}_{n})$ is the $n\times n$ identity matrix with columns
${\bf e}_1,~{\bf e}_2,\ldots,{\bf e}_{n}$.
$J=J_n=({\bf e}_n~|~{\bf e}_{n-1}~|\ldots~|~{\bf e}_{1})$ is the $n\times n$ reflection matrix, $J^2=I$.
$O_{k,l}$ is the $k\times l$ matrix filled with zeros.
A matrix $Q$ is called
{\em orthogonal} (also  {\em unitary} and {\em orthonormal}) if $Q^TQ=I$ or $QQ^T=I$.
%Let us define unitary matrix $Q(M)$ where $M$ has full column rank.
%------------------------------------------------------------------------------

\begin{theorem}\label{thqrf} \cite[Theorem 5.2.2]{GL96}.
A matrix $M$ of full column rank has unique
QR factorization $M=QR$ where $Q=Q(M)$
is an orthogonal matrix
and $R=R(M)$ is a square upper triangular matrix
with positive diagonal entries.
\end{theorem}

We use the matrix norms $||\cdot||_h$ for $h=1,2,\infty$
\cite[Section 2.3]{GL96}
and write $||\cdot||=||\cdot||_2$.
%\begin{equation}\label{eqnrm}
%$||A||^2\le ||A||_1||A||_{\infty}$
%\end{equation}
% for all matrices $A$.
We write $a\approx 0$ and $A\approx O$ 
if the values $|a|$ and $||A||$ are small in context.
We write $a\approx b$ for $b\neq 0$ and $A\approx B$
and $B\neq O$ if the ratios $|a|/|b|$ and $||A||/||B||$
are small.

$M^+$ is the Moore--Penrose pseudo inverse of $M$ \cite[Section 5.5.4]{GL96}.
An $n\times m$ matrix $X=M^{(I)}$ is a left (resp. right)
inverse of an $m\times n$  matrix $M$ if $XM=I_n$
(resp. if $MY=I_m$).
 $M^+$ is an $M^{(I)}$ for a matrix
$M$ of full rank.
 $M^{(I)}=M^{-1}$ for
 a nonsingular matrix $M$.

%------------------------------------------------------------------------------

%\begin{theorem}\label{thschur} {\rm Schur Decomposition} \cite[Theorem 7.1.3]{GL96}.
%Every $n\times n$ matrix $A$ can be represented
%as $Q^HRQ$ for a unitary matrix $Q$ and an upper triangular matrix $R$,
%so that the spectrum of $A$ is given by $\diag(R)$.
%\end{theorem}

%------------------------------------------------------------------------------

\medskip
{\bf Matrix computations: eigenspaces} \cite{GL96}, \cite{S01}, \cite{W02}, \cite{W07}, \cite{BDD00}.
$\mathcal S$
is an {\em invariant subspace} or {\em eigenspace} of a square matrix $M$
if $M\mathcal S=
\{M{\bf v}:{\bf v}\in \mathcal S\}\subseteq\mathcal S$.

%------------------------------------------------------------------------------

\begin{theorem}\label{thsubs} \cite[Theorem 4.1.2]{S01},
\cite[Section 6.1]{W02},
 \cite[Section 2.1]{W07}.
Let $U\in \mathbb C^{n\times r}$ be a matrix basis for
an eigenspace  $\mathcal U$ of a matrix $M\in\mathbb C^{n\times n}$.
Then the matrix $L=U^{(I)}MU$ is unique (that is independent
of the choice of the left inverse $U^{(I)}$) and satisfies
$MU=UL$.
%Furthermore, $\{\lambda,  \Phi\}$ is an eigenpair
%of $L$ if and only if $\{\lambda,  X\Phi\}$ is an eigenpair of $A$.
\end{theorem}

%------------------------------------------------------------------------------
%------------------------------------------------------------------------------

The above pair $\{L,\mathcal U\}$ is an {\em eigenpair}
of a matrix $M$, $L$ is its {\em eigenblock},
and $\mathcal U$ is the {\em associated eigenspace} of $L$ \cite{S01}.
%Every subspace of $\mathcal U$
%is also an eigenspace associated with the eigenblock $L$.
%A scalar $\lambda$ is  an {\em eigenvalue} of $M$.
 If $L=\lambda I_n$, then also
$\{\lambda,  \mathcal U\}$ is called
an {\em eigenpair} of a matrix $M$.
 In this case $\det(\lambda I-M)=0$,
 whereas $\mathcal N(M-\lambda I)$
is the eigenspace associated with the {\em eigenvalue} $\lambda$
and made up of
its {\em eigenvectors}.
$\Lambda(M)$ is the set of all eigenvalues of $M$,
called its {\em spectrum}.
$\rho(M)=\max_{\lambda\in \Lambda(M)} |\lambda|$
is the {\em spectral radius} of $M$.
Theorem \ref{thsubs} implies that $\Lambda(L)\subseteq \Lambda(M)$.
For an eigenpair $\{\lambda,  \mathcal U\}$
write $\psi=\min~|\lambda/\mu|$ over
$\lambda\in \Lambda(L)$ and $\mu \in \Lambda(M)-\Lambda(L)$.
Call
 the eigenspace $\mathcal U$ {\em dominant}
if $\psi>1$,
{\em dominated} if $\psi<1$, {\em strongly dominant}
if $1/\psi\approx 0$, and {\em strongly dominated}
if $\psi\approx 0$.
An $n\times n$ matrix $M$ is called {\em diagonalizable or nondefective}
if $SMS^{-1}$ is a  diagonal matrix for some matrix $S$, e.g., if
$M$ has $n$ distinct
eigenvalues. A random real or complex perturbation makes the matrix
 diagonalizable with probability 1.
{\em In all our algorithms we assume
 diagonalizable input matrices.}

%------------------------------------------------------------------------------

\begin{theorem}\label{thsmf} (See \cite[Theorem 1.13
%(c),(d)
]{H08}.)
$\Lambda(\phi(M))=\phi(\Lambda(M))$
for a square matrix $M$ and a function $\phi(x)$ defined on its spectrum.
Furthermore $(\phi(\lambda),\mathcal U)$ is an eigenpair of the matrix $\phi(M)$
if the matrix $M$ is diagonalizable and has an eigenpair $(\lambda,\mathcal U)$.
\end{theorem}

%------------------------------------------------------------------------------

A nonsingular matrix $M$ is {\em well conditioned}
if its condition number $\kappa (M)=||M||~||M^{-1}||\ge 1$ is reasonably bounded.
This matrix is
 {\em ill conditioned}
if its condition number is large. $\kappa (M)=||M||=||M^+||=1$
for orthogonal matrices $M$. 

%------------------------------------------------------------------------------
%------------------------------------------------------------------------------
% the eigenvalues of a diagonal or triangular matrix are given by its diagonal entries.
\medskip
%------------------------------------------------------------------------------
{\bf Toeplitz matrices} \cite[Ch. 2]{P01}.
%------------------------------------------------------------------------------
An $m\times n$ Toep\-litz matrix 
$T=(t_{i-j})_{i,j=1}^{m,n}$
%(resp. Hankel matrices $H=(h_{i+j})_{i,j=1}^{m,n}$)
is defined by the $m+n-1$ entries of its first row and column,
in particular
$$T=(t_{i-j})_{i,j=1}^{n,n}=\begin{pmatrix}t_0&t_{-1}&\cdots&t_{1-n}\\ t_1&t_0&\smallddots&\vdots\\ \vdots&\smallddots&\smallddots&t_{-1}\\ t_{n-1}&\cdots&t_1&t_0\end{pmatrix}.$$

%------------------------------------------------------------------------------

\medskip
{\bf Polynomials and companion matrices.} Write
\begin{equation}\label{eqpol}
p(x)=\sum_{i=0}^np_ix^{i}=p_n\prod_{j=1}^n(x-\lambda_j),
\end{equation}
\begin{equation}\label{eqrev}
p_{\rm rev}(x)=x^np(1/x)=\sum_{i=0}^np_ix^{n-i}=p_n\prod_{j=1}^n(1-x\lambda_j),
\end{equation}
$p_{\rm rev}(x)$ is the reverse polynomial of $p(x)$,
$$C_p=\begin{pmatrix}
        0   &       &       &   & -p_0/p_n \\
        1   & \ddots    &       &   & -p_1/p_n \\
            & \ddots    & \ddots    &   & \vdots    \\
            &       & \ddots    & 0 & -p_{n-2}/p_n \\
            &       &       & 1 & -p_{n-1}/p_n \\
    \end{pmatrix},~{\rm for}~{\bf p}=(p_j)_{j=0}^{n-1},$$
and $C_{p_{\rm rev}}=JC_pJ$ are the $n\times n$  companion matrices
of the polynomials
$p(x)=\det(xI_n-C_p)$ and $p_{\rm rev}(x)=\det(xI_n-C_{p_{\rm rev}})$,
 respectively.

%------------------------------------------------------------------------------

\begin{fact}\label{facfr} (See \cite{C96} or \cite{P05}.)
The companion matrix $C_p\in \mathbb C^{n\times n}$ of a polynomial $p(x)$ of (\ref{eqpol})
generates an algebra $\mathcal A_p$ of matrices
having structure of Toeplitz type.
One needs  $O(n)$ flops for  addition,
$O(n\log n)$ flops for multiplication and $O(n\log^2 n)$ flops  
for inversion  in  this algebra 
and needs $O(n\log n)$ flops for multiplying a matrix from the algebra  
by a square Toeplitz matrix.
\end{fact}

%------------------------------------------------------------------------------

\section{%Approximating selected ei\-gen\-va\-lues
Basic algorithms for approximating selected ei\-gen\-va\-lues}\label{ssse}

%------------------------------------------------------------------------------

The following algorithms employ Theorems \ref{thsubs} and \ref{thsmf} to approximate
a  specified set $\widehat \Lambda$
of the eigenvalues of a matrix,
e.g., 
its absolutely largest eigenvalue or 
 the set of its real eigenvalues.
They will serve as the basis for our 
eigenvalue algorithms,
which we will apply to the companion matrices
 in the subsequent sections.

%------------------------------------------------------------------------------
%------------------------------------------------------------------------------

\begin{algorithm}\label{fl1} {\bf  Reduction of the input
size for ei\-gen-solving.}
%------------------------------------------------------------------------------

\begin{description}

%------------------------------------------------------------------------------

\item[{\sc Input:}] a diagonalizable matrix $M\in \mathbb R^{n\times n}$
and a property that specifies a subset $\Lambda$ of its unknown spectrum
associated with an unknown eigenspace $\mathcal U$.
%e.g., the set of all its real or
%absolutely largest eigenvalues.

%------------------------------------------------------------------------------

\item[{\sc Output:}] two matrices
 $\widehat L$ and $\widehat U$ such 
that the pair $\{\Lambda(\widehat L),\mathcal R(\mathcal U)\}$
closely approximates the eigen\-pair
$\{\Lambda,\mathcal U\}$ of the matrix $M$.

%-----------------------------------------------------------------------------

%\item[{\sc Initialization:}]
%Fix a sufficiently large integer  $h$.

%------------------------------------------------------------------------------

\item[{\sc Computations:}] $~$

%------------------------------------------------------------------------------
\begin{enumerate}

%------------------------------------------------------------------------------

\item %1
Compute a  matrix function $\phi(M)$ for which the linear space 
$\mathcal U$ is a strongly
dominant  ei\-gen\-space.
%of the targit eigenpair $\{L,\mathcal U\}$.
\item %2
Compute and output a matrix $\widehat U$ of full column rank
 whose range approximates the
ei\-gen\-space $\mathcal U$.
% the Subspace Iteration \cite[Section 6.1]{S01}
%One can refine the computed approximations
%by applying the
%Inverse Power method or its extensions into the Inverse Orthogonal Iteration
%\cite[page 339]{GL96}) or Subspace Iteration \cite[Section 6.1]{S01}.)
%We can initiate these iterations with
%a real vectors or a real matrix and
%delete the imaginary parts of the entries
%of all computed auxiliary vectors and matrices
%(by virtue of Fact \ref{facreal}
%they represent just rounding errors),
%and output the resulting matrix $U$.
\item %3
Compute the left inverse $\widehat U^{(I)}$
of the matrix $\widehat U$.
\item %4
Compute and output the matrix $\widehat L=\widehat U^{(I)}M\widehat U$.
\end{enumerate}

%------------------------------------------------------------------------------

\end{description}

%------------------------------------------------------------------------------

\end{algorithm}

%------------------------------------------------------------------------------

At Stage 2 of the algorithm,
 one can apply a
rank revealing QR or LU factorization of the matrix $\phi(M)$ \cite{GE96},
\cite{HP92}, \cite{P00a}.
Given a reasonably close upper bound $r_+$ on the dimension $r$
of the eigenspace $\mathcal U$,
 we can 
alternatively
 employ a randomized multiplier as follows.

%------------------------------------------------------------------------------

\begin{algorithm}\label{fleigsp} {\bf  Approximation of a dominant eigenspace.}

%-----------------------------------------------------------------------------

\begin{description}

%-----------------------------------------------------------------------------
\item[{\sc Input:}] a positive integer $r_+$ and a diagonalizable matrix
$W\in \mathbb R^{n\times n}$
that has numerical rank $n-r$ and has strongly dominant
eigenspace $\mathcal U$ of dimension $r>0$ for an unknown
$r\le r_+$.

%------------------------------------------------------------------------------

\item[{\sc Output:}] an $n\times r$   matrix $\widehat U$ 
%of full rank $r$
such that $\mathcal R(\widehat U)\approx\mathcal U$.

%------------------------------------------------------------------------------

\item[{\sc Computations:}] $~$

%------------------------------------------------------------------------------
\begin{enumerate}

%------------------------------------------------------------------------------

\item %1
Compute the $n\times r_+$ matrix $WG$ for a well conditioned random $n\times r_+$ 
matrix $G$.
\item %2
Compute the rank revealing QR  factorization of the matrix $WG$ and
 output an orthogonal matrix basis $\widehat U$ of this matrix.
\end{enumerate}

%------------------------------------------------------------------------------

\end{description}

%------------------------------------------------------------------------------

\end{algorithm}

%------------------------------------------------------------------------------

The algorithm amounts to a single iteration of the Power Method
\cite{GL96}, \cite{S01}. This is expected to be sufficient where the matrix 
$W$ has a strongly dominant eigenspace. By virtue of Fact \ref{facfr}  
we would benefit from choosing a random Toeplitz multiplier $G$
where the matrix $W$ belongs to the matrix algebra $\mathcal  A_p$,
generated by the companion matrix $C_p$ of a polynomial $p(x)$.
According to the study in \cite{PQa} Gaussian 
random Toeplitz matrices
are likely to be reasonably well conditioned  
under both standard Gaussian 
and  uniform
probability distribution.

%Clearly, $\rank(WG)=n-r$ with probability 1.
%Define the matrix $\tilde W$ by zeroing the $r$
%smallest singular values of $W$. 
%Then $\tilde W\approx W$ because $\sigma_{n-r+1}(W)$
%is small, and so
% $\tilde WG\approx WG$ and $\mathcal R(\tilde W)\approx \mathcal U$.
%Deduce from
%Theorem \ref{thgh} that $\mathcal R(\tilde WG)\approx \mathcal R(\tilde W)$.
%Finally combine  the latter two  relationships and obtain that
%$\mathcal R(\tilde WG)\approx \mathcal U$.

%\begin{remark}\label{retpl}
%If $W=W(C_p)$ is a rational matrix function 
%and if the integer $r_+$ is not small,
%we can choose matrix $G\in \mathcal T^{n\times r_+}$,
%rather than  $G\in \mathcal G^{n\times r_+}$
%in Algorithm \ref{fleigsp}. In this case we 
% multiply $W$ by $G$
%by using $O(n\log n)$ flops 
%rather than order of $n^2r$
%flops (see Fact \ref{facfr}). 
%The results of Section \ref{scgrm}
%displayed in Table \ref{tabletpl}
%and to some extent the results in \cite{PQa}
%provide emprical evidence that
%this choice of the multiplier $G$ is
%still expected to support Algorithm \ref{fleigsp}.
%\end{remark}
%We can readily estimate
% the absolute values of these eigenvalues cannot exceed
%$\min \{||B||_1,||B||_{\infty}\}$, and so
%the matrix $B+bI$ is positive

%------------------------------------------------------------------------------

Now assume a  nonsingular
 matrix $\tilde \phi(M)$ with
a dominated
(rather than dominant)
eigenspace  $\mathcal U$.
Then this is a dominant
eigenspace of the matrix
 $(\tilde \phi(M))^{-1}$. We can 
apply Stages 2--4 of Algorithm \ref{fl1}
to this eigenspace or, alternatively, apply
the following
%dual 
variation of Algorithm \ref{fl1}.

%------------------------------------------------------------------------------

\begin{algorithm}\label{fl1d} {\bf  Dual reduction of the
input size for eigen-solving.}

%------------------------------------------------------------------------------

\begin{description}

%----------------------------------------niwsniws--------------------------------------

\item[{\sc Input}, {\sc Output}]
and Stages 3 and 4 of {\sc Computations}
as in Algorithm \ref{fl1}.

%-----------------------------------------------------------------------------

%\item[{\sc Initialization:}]
%Fix a sufficiently large integer  $k$.

%------------------------------------------------------------------------------

\item[{\sc Computations:}] $~$

%------------------------------------------------------------------------------
\begin{enumerate}

%------------------------------------------------------------------------------

\item %1
Compute a  matrix function $\phi(M)$ having strongly
dominated  ei\-gen\-space  $\mathcal U$.
\item %2
Apply the Inverse Orthogonal Iteration \cite[page 339]{GL96}
to the matrix $\phi(M)$ to
output a matrix $\widehat U$ of full column rank
 whose range approximates the
ei\-gen\-space $\mathcal U$. Output the matrix $\widehat L=\widehat U^{(I)}M\widehat U$.
%\item %3,4
%and 4. as in Flowchart \ref{fl1}.
\end{enumerate}

%------------------------------------------------------------------------------

\end{description}

%------------------------------------------------------------------------------

\end{algorithm}

%------------------------------------------------------------------------------

\begin{remark}\label{repow}
Seeking a single eigenvalue of a matrix $M$ and
 having performed Stage 1 of Algorithm \ref{fl1}
(resp. \ref{fl1d}),
we can apply the Power (resp. Inverse Power) Method
(cf. \cite[Sections 7.3.1 and 7.6.1]{GL96}, \cite{BGP02/04}) to
approximate an eigenvector
${\bf v}$
of the matrix $\phi(M)$ in its dominant (resp. dominated)
eigenspace  $\mathcal U$. This eigenvector is shared with the  matrix $M$
by virtue of Theorem \ref{thsmf},
and we can approximate the associated eigenvalue  by  the
Rayleigh quotient ${\bf v}^TM{\bf v}/{\bf v}^T{\bf v}$
or a simple quotient 
${\bf v}^TM{\bf e}_j/{\bf v}^T{\bf e}_j$
for a fixed or random integer $j$, $1\le j\le n$,
in \cite{BGP02/04},  \cite{PQZC} and \cite{PZ10/11}.
We can employ deflation
or reapply our algorithm for other
initial approximations
  (cf. our Section \ref{sniws} and \cite{HSS01})
to approximate other
eigenvalues of the  matrix $M$.
\end{remark}

%------------------------------------------------------------------------------

\begin{remark}\label{renum}
In numerical implementation of the algorithms of this section
one should compute a matrix basis $L_+$
for the dominant (resp. dominated) eigenspace $\mathcal U_+$ of the
matrix $\phi_+(M)$ (resp. $\tilde \phi_+(M)$) such that
 $\mathcal U_+\supseteq \mathcal U$ and has a dimension $r_+\ge r$.
Then the matrix $L_+$ has the size $r_+\times r_+$
and shares $r$ desired and $r_+-r$ extraneous eigenvalues  with the matrix $M$.
For example, in numerical real eigen-solving the eigenspace
$\mathcal U_+$ is associated with
all real and nearly real eigenvalues of $M$, 
and having them approximated we can readily select among them the $r$ real 
eigenvalues.
\end{remark}

%------------------------------------------------------------------------------

In the next sections we describe some algorithms for
computing the  matrix functions $\phi(M)$ and $\tilde \phi(M)$
at Stages 1 of Algorithms \ref{fl1} and \ref{fl1d}.
%towards complex and real eigen-solving.

%------------------------------------------------------------------------------

\section{The computation of the dominant eigenspaces by means of repeated squaring,
shifts and inversions}\label{srs}

%------------------------------------------------------------------------------

\subsection{Repeated squaring in the Frobenius algebra
with simplified recovery of the eigenvalues}\label{srssr}

%------------------------------------------------------------------------------

Theorem \ref{thsmf} for $\phi(M)=M^k$ implies
 that for a diagonalizable matrix $M$ and
  sufficiently large integers $k$, the matrices $M^k$
have dominant eigenspace $\mathcal U$ associated with the
set of the absolutely largest eigenvalues of $M$.
%\cite[Section 7.3]{GL96}, \cite[Section 6.1.1]{S01}.
For a fixed or random real or complex shift $s$ 
we can 
%ompute high powers $M^k$ at quite a low cost by 
write $M_0=M-sI$ and compute $M_0^{2^h}$
in $h$ 
squarings,
\begin{equation}\label{eqrs}
M_{h+1}=a_hM_h^2,~a_h\approx 1/||M_h||^2~{\rm for}~h=0,1,\dots
\end{equation}
Suppose $M$ is a real diagonalizable matrix with simple eigenvalues
and $h$ is a reasonably large integer.
Then 
with probability 1
the dominant eigenspace $\mathcal U$ of $M_{h}$
 has dimension 1  for random nonreal shifts $s$
and has dimension 1 or 2 for a random real $s$.
If the matrix $M$ has a single absolutely largest 
eigenvalue of multiplicity $m$ or has a 
cluster of $m$ simple absolutely largest
eigenvalues, then the associated eigenspace of dimension $m$ 
is dominant for the matrix $M_h$ and a reasonably large integer $h$.
As in the case of Algorithm \ref{fleigsp},
the column space of the product $M_hG$ for a random well conditioned 
$n\times m$ matrix $G$ is expected to approximate this eigenspace.

For $M=C_p$ we can follow  \cite{C96} and apply the FFT-based
algorithms that support Fact \ref{facfr}
to perform every squaring and every multiplication
in $O(n\log n)$ flops.
The bottleneck of that paper and its amelioration in \cite{P05}
 is the
recovery of the roots of $p(x)$ 
%of (\ref{eqpol})
%by symbolic methods
 at the end of the squaring process where
$|\lambda_j|\approx |\lambda_k|$ for $j\neq k$. The paper
\cite{P05} relieves some difficulties of  \cite{C96}
by employing approximations to the roots of  $p'(x)$, $p''(x)$, etc., but
these techniques are still too close to
the symbolic recovery methods of the paper \cite{C96},
which operates with polynomials and does not employ numerical 
linear algebra.
In contrast Algorithms \ref{fl1} and  \ref{fl1d}
reduce the computation of the $r$ eigenvalues of a  selected
subset of the spectrum $\Lambda (M)$
 to eigen-solving for the $r\times r$
matrix $L$, and this is simple where  $r$ is a small integer.
Now replace $M_0$ in (\ref{eqrs}) by $M_0=(M-\sigma I)^{-1}$ for a fixed complex $\sigma$.
Then the above algorithms approximate the dominant eigenspace of the matrix $M_h$ for 
a large integer $h$ and
 the associated  set of the eigenvalues of $M$, which are
the nearest to the point $\sigma$. E.g., this is the set of 
the absolutely smallest eigenvalues where $\sigma=0$.
For $M=C_p$ we can alternatively write
$M_0=C_{p_{\rm rev}(x-\sigma)}$ in (\ref{eqrs})
to replace the inversion of the shifted companion matrix
with Taylor's shift of the variable $x$ of the polynomial $p(x)$
and the reversion of the order of its coefficients.

%------------------------------------------------------------------------------

\subsection{Approximation of the real eigenvalues: basic results}\label{sbrslt}

%------------------------------------------------------------------------------

Next we  map  the complex plane
to transform the real line into the unit circle $\{z:~|z|=1\}$
and then apply
repeated squaring, which maps the unit circle into itself and
 sends  the image of any nonreal eigenvalue of the input matrix 
towards 0 or $\infty$,
thus ensuring desired isolation of the images.
%lying on the unit circle. 

%------------------------------------------------------------------------------

\begin{fact}\label{facrcsqr}
Write $\lambda=u+v\sqrt{-1}$,
\begin{equation}\label{eqmu}
\mu=(\lambda +\sqrt{-1})(\lambda -\sqrt{-1})^{-1},~\beta_k=\frac{\sqrt{-1}(\mu^k+1)}{\mu^k-1}
\end{equation}
for a positive integer $k$.
Then

\medskip

\noindent ${\rm (a)}~~~~~~\beta_0=\lambda=\frac{\sqrt{-1}(\mu+1)}{\mu-1},$

\medskip

\noindent ${\rm (b)}~~~~~~\mu=\frac{n(\lambda)}{d(\lambda)}$
for $n(\lambda)=u^2+v^2-1+2u\sqrt{-1}$  and $d(\lambda)=u^2+(v-1)^2$,
and consequently

\medskip

\noindent ${\rm (c)}~~~~~~|\mu|^2=\frac{(v^2-1)^2+(u^2 +2v^2+1)u^2}{(u^2+(v-1)^2)^2}$,

\medskip

\noindent ${\rm (d)}~~~~~~|\mu|=1~{\rm if~and~only~if}~ \lambda~{\rm is~real.}$

\medskip

Furthermore

\medskip

\noindent ${\rm (e)}~~~~~~\beta_k=\frac{n_k(\lambda)}{d_k(\lambda)}~~~$ for
$~~~n_k(\lambda)=\sum_{g=0}^{\lfloor k/2\rfloor}(-1)^g\begin{pmatrix}k\\2g\end{pmatrix}\lambda^{k-2g}~~~$ and
%\begin{equation}\label{eqnd}
$$d_k(\lambda)=\sum_{g=0}^{\lfloor k/2\rfloor}(-1)^{g+1}\begin{pmatrix}k\\2g+1\end{pmatrix}\lambda^{h-2g-1}.$$
%\end{equation}
\end{fact}

Fact \ref{facrcsqr} implies that the transform $\lambda\rightarrow \mu$ maps the real line
onto the unit circle $\mathcal C_1=\{\mu:|\mu|=1\}$.
Powering of the value $\mu$ keeps this circle in place,
whereas the transform $\mu^k\rightarrow \beta_k$ 
moves it back to the real line. Furthermore
values $|\mu|^k$ converge to 0 for $|\mu|<1$ and to $+\infty$
for $|\mu|>1$ as $k\rightarrow \infty$. Therefore for large $k$
the transform  $\mu^k\rightarrow \beta_k$ sends
the images of the  nonreal values $\lambda$ into some neubourhood
of the values $\sqrt{-1}$ and $-\sqrt{-1}$.
Then the transform  $\beta_k\rightarrow \gamma_k=\beta_k^2+1$
sends these images into the neighborhood of
the origin, whereas the real eigenvalues  $\beta_k$ are moved into the real 
values $\gamma_k\ge 1$. This enables the desired domination of 
the images of the real eigenvalues of the matrix $M$ over the images of its nonreal
 eigenvalues. We can recover  the eigenvalues $\lambda_k$ of 
the matrix $M$ as soon as we approximate their eigenspaces
shared with the eigenspaces associated with the eigenvalues 
$\gamma_k$
of the matrices 
%------------------------------------------------------------------------------

\begin{equation}\label{eqqmi}
Q_k=M_k^2+I_n
\end{equation}
where 
\begin{equation}\label{eqp}
P=(M+I_n\sqrt {-1})(M-I_n\sqrt {-1})^{-1},
\end{equation}
\begin{equation}\label{eqmk}
M_k=\sqrt{-1}(P^k+I_n)(P^k-I_n)^{-1},
\end{equation}
and in particular
$M_1=M$, whereas
$2M_2=M-M^{-1}$.
\begin{corollary}\label{corcs}
Suppose that an $n\times n$ matrix  $M$ has exactly
$s$ eigenpairs $\{\lambda_j,\mathcal U_j\}$,
$j=1,\dots,s$,
and does not have eigenvalues $\pm \sqrt {-1}$.
Assume the equations of Fact \ref{facrcsqr}
as well as equations (\ref{eqqmi})--(\ref{eqmk}).
Furthermore 
write
$$\beta_{j,k}=\frac{n_k(\lambda_{j})}{d_k(\lambda_{j})},
~~n_k(\lambda_{j})=\sum_{g=0}^{\lfloor k/2\rfloor}(-1)^g\begin{pmatrix}k\\2g\end{pmatrix}\lambda_j^{k-2g},
~d_k(\lambda_{j})=\sum_{g=0}^{\lfloor k/2\rfloor}(-1)^{g+1}\begin{pmatrix}k\\2g+1\end{pmatrix}\lambda_j^{k-2g-1},
$$ and
$\mu_j=(\lambda_j+\sqrt {-1})(\lambda_j-\sqrt {-1})^{-1}$ for $j=1,\dots,s$.
Then 
$M_k=n_k(M)(d_k(M))^{-1}$ where

$$n_k(M)=\sum_{g=0}^{\lfloor k/2\rfloor}(-1)^g\begin{pmatrix}k\\2g\end{pmatrix}M^{k-2g},~
d_k(M)=\sum_{g=0}^{\lfloor k/2\rfloor}(-1)^{g+1}\begin{pmatrix}k\\2g+1\end{pmatrix}M^{h-2g-1},$$
 and
the matrices $Q_k=M_k^2+I_n$ 
of (\ref{eqqmi})
have the eigenpairs
$\{\{\beta_{j,k},\mathcal U_j\},~j=1,\dots,s\}$
where
 $ \beta_{j,k}$ are real and $\beta_k\ge 1$ if
  $\lambda_j$ is real,
$\beta_{j,k}\rightarrow 0$ as $k\rightarrow \infty$ unless
$\lambda_j$ is real.
\end{corollary}

%------------------------------------------------------------------------------

\subsection{Approximation of the real eigenvalues: the algorithm}\label{srealg}

%------------------------------------------------------------------------------

The corollary   
suggests setting $\phi(M)=Q_k$ in Algorithm \ref{fl1} where the 
integers $k$ are sufficiently large.
 We can apply repeated squaring
to compute high powers $P^k$. 
%for $k=2^h$.
In numerical implementation we should
apply  scaling to avoid
large
 norms $||P^k||_q$.

Below is an algorithm that implements this approach
by using only two matrix inversions; this is
 much less than in
iteration (\ref{eqsignr}).
The algorithm works for a large class of inputs $M$, although it
can fail for harder inputs $M$,
which have many real and nearly real eigenvalues, but also 
have some other nonreal
eigenvalues. 
%Typically such hard matrices remain hard for any
% real parameters $a$, $t$, $v$,  and $w$  and transform $M\rightarrow \frac{M+tI_n}{vM+wI_n}$
%(see Part (c) of Fact \ref{facrcsqr}).
The heuristic choice
\begin{equation}\label{eqatvw}
v=0,~w=1,~t\approx-\Re(\trace (M)),~a=\frac{t}{n},~{\rm and}~\widehat M=M+tI_n
\end{equation}
 tends to push the values $|\mu|$
away from $1$ on the average
input,
motivating application of the algorithm to the input 
matrix
$\widehat M$ rather than $M$,
 although this shift can strongly push the value 
$|\mu|$
toward 1 for the worst case
input. Note that trace$(M)$ is a real value where $M$ is a real matrix.

%------------------------------------------------------------------------------

\begin{algorithm}\label{fl2} {\bf  Mapping the real line onto the unit circle and
repeated squaring} 
%(cf. Remark \ref{restit5}).
%------------------------------------------------------------------------------

\begin{description}

%------------------------------------------------------------------------------

\item[{\sc Input:}] a real $n\times n$ matrix $M$,
whose real and nearly real eigenvalues
are associated with an unknown  eigenspace $\mathcal U_+$
having an unknown dimension $r_+\ll n$.

%------------------------------------------------------------------------------

\item[{\sc Output:}] FAILURE or a  matrix $\widehat U$
such that  $\mathcal R(\widehat U)\approx \mathcal U_+$.

%-----------------------------------------------------------------------------

\item[{\sc Initialization:}]
Fix sufficiently large tolerances $\tau$
and $h_+$, fix real $a$, $t$, $v$, and $w$
and the
 matrix $\widehat M$ of (\ref{eqatvw}).

%------------------------------------------------------------------------------

\item[{\sc Computations:}] $~$

%------------------------------------------------------------------------------
\begin{enumerate}

%------------------------------------------------------------------------------

\item %1

Compute the matrices $P=(a\widehat M+I_n\sqrt {-1})(a\widehat M-I_n\sqrt {-1})^{-1}$
(cf. Corollary \ref{corcs}) and $P^{2^{g}}$ for $g=1,2,\dots,h+1$
until $||P^{2^{h+1}}||_q>\tau$ for a fixed $q$ (e.g., for $q=1$ or $q=\infty$)
or until $h\ge h_+$.
\item %2
Compute matrix $M_k$
of Corollary \ref{corcs} for $k=2^{h_+}$.
\item %32^h
Apply Algorithm \ref{fleigsp} to the matrix $\phi=Q_k$
and the integer $r=n$ to output an $n\times r$
matrix basis for the strongly dominant eigenspace
$\widehat U$ of $F$. 
\item %4
Output FAILURE if Algorithm \ref{fleigsp} fails, which would mean that 
the matrix $\phi=Q_k$ has no strongly dominant eigenspace of dimension $r_+< n$.

\end{enumerate}

%------------------------------------------------------------------------------

\end{description}

%------------------------------------------------------------------------------

\end{algorithm}

%------------------------------------------------------------------------------

\begin{remark}\label{restit5}
We can compute the matrix $P^k$ for a sufficiently large integer $k=2^{h_+}$ to ensure
isolation of the images of real and nearly real eigenvalues of $M$
from the images of its other eigenvalues and then,
as an alternative to the application of Algorithm \ref{fleigsp} 
at Stage 3, we  can apply
the Rayleigh Quotient Iteration to the matrix $P^k$ to 
approximate the associated eigenspace shared by the matrices $P^k$ and $M$.
\end{remark}

%------------------------------------------------------------------------------

\begin{remark}\label{resqrng}
We can modify Stage 4 to compute an integer $h_+$ iteratively,
according to a fixed policy: we can begin with a small  $h_+$,
then increase it, and reapply the algorithm if the computations fail.
Alternatively we can estimate the integer $h_+$ a priori
 if we estimate
the absolute values of all eigenvalues of the matrix
$P$ by
computing its Gerschg{\"o}rin discs
\cite[page 320]{GL96}, \cite[page 39]{S01}
(see also the end of the Appendix).
\end{remark}

%------------------------------------------------------------------------------

\subsection{Modification by using
% epeated squar\-ing and 
the M{\"o}bius 
transform}\label{scnrs}

%------------------------------------------------------------------------------

In an alternative
iteration we begin in the same way 
as Algorithm \ref{fl2} but
interrupt repeated squaring
 by applying
the scaled M{\"o}bius transform
$P^k\rightarrow P^k+P^{-k}$
instead of the maps $P\rightarrow M_k$ of (\ref{eqmk})
and $M_k\rightarrow Q_k=M_k^2+I_n$ of (\ref{eqqmi}).
The scaled M{\"o}bius transform
moves the images of
all real eigenvalues of the matrix  $M$ from  the unit circle
$\mathcal C_1$
into the real line interval $[-2,2]$,
whereas for reasonably large integers $k$ it moves the other eigenvalues 
into the exterior 
 of the disc $D_{8/3}(0)$.
(Namely the map $M\rightarrow P^k$
moves the nonreal eigenvalues of the matrix $M$
towards 0 or $\infty$ and thus  for reasonably large integers $k$
moves them into the exterior  
of the annulus
 $\mathcal A_{1/3,3}(0)=\{x:~1/3\le |x|\le 3\}$,
which the scaled M{\"o}bius transform
$P^k\rightarrow P^k+P^{-k}$
moves 
into the exterior  
 of the disc $D_{8/3}(0)$.)
Consequently by using 
the map $M\rightarrow P^k+P^{-k}$ we
isolate from one another
the two sets of the real and nonreal eigenvalues of the input companion matrix
$M$. Then we   
make the eigenspace associated with real eigenvalues of 
the matrix $M$
 dominated or dominant simply by
squaring  reasonably many times
 the  matrix $P^k+P^{-k}$ or its inverse,
 respectively,
  and then it remains to
 apply Algorithm \ref{fl1d} (respectively \ref{fl1})
to approximate these eigenvalues.
The images of some real eigenvalues of the matrix $M$
 dominated by the images of other of them 
 would be lost numerically 
due to rounding errors
unless we apply orthogonalization
or deflation. 
Next we prove the stated properties of this combination of 
 the
maps of Fact \ref{facrcsqr}, repeated squaring, and
 the M{\"o}bius transform.

%------------------------------------------------------------------------------

\begin{fact}\label{facrtoc} (Cf. Fact \ref{facrcsqr} for $a=1$.)
Write
\begin{equation}\label{eqmu0}
\mu=(\lambda +\sqrt{-1})(\lambda -\sqrt{-1})^{-1}.
\end{equation}
Then

(a) $\lambda=\sqrt{-1}(\mu-1)/(\mu+1)$,

(b) $|\mu|=1$ if and only if $\lambda$ is real and

(c) $\mu_k=\mu^k+\mu^{-k}=
\sum_{g=0}^{k}(-1)^g\begin{pmatrix}2k\\2g\end{pmatrix}\lambda^{2k-2g}(\lambda^2+1)^{-k}$
for $k=1,2,\dots$.
(In particular $\mu_1=\frac{\lambda^{2}-1}{\lambda^{2}+1}$,
whereas $\mu_2=\frac{\lambda^{4}-6\lambda^{2}+1}{(\lambda^{2}+1)^2}$.)
\end{fact}

\begin{fact}\label{facinf}
Assume $\mu$ of (\ref{eqmu0}) and a nonnegative integer $k$.
Then $|\mu|=1$ and $-2\le \mu^k+\mu^{-k}\le 2$
if $\lambda$ is real,
whereas $|\mu^k+\mu^{-k}|\rightarrow \infty$ as $k\rightarrow \infty$
otherwise.
\end{fact}

\begin{corollary}\label{coinf}
Assume  that an $n\times n$ matrix  $M$ has exactly
$s$ eigenpairs $\{\lambda_j,\mathcal U_j\}$,
$j=1,\dots,s$,
and does not have eigenvalues $\pm \sqrt {-1}$.
By extending (\ref{eqp})  and (\ref{eqmu0}),
write $$P=(M+I_n\sqrt {-1})(M-I_n\sqrt {-1})^{-1}=(M-I_n\sqrt {-1})^{-1}(M+I_n\sqrt {-1}),$$
\begin{equation}\label{eqtk}
T_k=P^{k}+P^{-k}=\sum_{g=0}^{k}(-1)^g\begin{pmatrix}2k\\2g\end{pmatrix}M^{k-2g}(M^2+1)^{-k},
\end{equation}
$$\mu_j=(\lambda_j+\sqrt {-1})(\lambda_j-\sqrt {-1})^{-1},$$
$$\mu_{j,k}=\mu_j^{k}+\mu_j^{-k}=
\sum_{g=0}^{k}(-1)^g\begin{pmatrix}2k\\2g\end{pmatrix}\lambda_j^{k-2g}(\lambda_j^2+1)^{-k}$$
for $k=1,2,\dots$ (In particular
$T_1=2(I_n -M^2)(I_n+M^2)^{-1}=2I_n-4(I_n+M^2)^{-1}$, whereas
$T_2=(M^4-6M^2+I_n)(M^2+I_n)^{-2}=(M^2+I_n)^{-2}(M^4-6M^2+I_n)$.)
Then $M=\sqrt{-1}(P-I_n)(P+I_n)^{-1}=\sqrt{-1}(P+I_n)^{-1}(P-I_n)$,
$\lambda_j=\sqrt{-1}(\mu_j-1)/(\mu_j+1)$ for $j=1,\dots,s$,
 and
the matrices $T_k$ have the eigenpairs
$\{\{\mu_{j,k},\mathcal U_j\},~j=1,\dots,s\}$
where
 $-2\le \mu_{j,k}\le 2$ if
  $\lambda_j$ is real,
$|\mu_{j,k}|\rightarrow \infty$ as $h\rightarrow \infty$ unless
$\lambda_j$ is a real value.
\end{corollary}

%------------------------------------------------------------------------------

\section{The computation of the dominant eigenspaces by approximating
the matrix sign function}\label{smsf}

%------------------------------------------------------------------------------

\subsection{The matrix sign function: definition and basic properties}\label{smsf1}

%------------------------------------------------------------------------------

\begin{definition}\label{defsign}
For two real numbers $x\neq 0$ and $y$,
the function
 $\sign (x +y\sqrt {-1})$ is
equal to $1$ if $x>0$ and is equal to $-1$ if $x<0$.
\end{definition}

%------------------------------------------------------------------------------

\begin{definition}\label{defmsign} (See \cite{H08}.)
Let $A=ZJZ^{-1}$ be a Jordan canonical decomposition
of an $n\times n$ matrix $A$ where $J=\diag(J_-,J_+)$,
$J_-$ is a $p\times p$ matrix
   and all its $p$ diagonal entries
have negative real parts, whereas
$J_+$ is a $q\times q$ matrix
   and all its $q$ diagonal entries
have positive real parts.
Then $\sign(A)=Z\diag(-I_p,I_q)Z^{-1}$.
 Equivalently
$\sign(A)=A(A^2)^{-1/2}$ or $\sign(A)=\frac{2}{\pi}A\int_0^{\infty}(t^2I_n+A^2)^{-1}dt$.
\end{definition}

%See \cite[Chapter 5]{H08} and the bibliography
%therein on the main properties of
%the matrix sign function.
% and the known algorithms for computing it.

%------------------------------------------------------------------------------

\begin{definition}\label{defgmsign}
Assume the matrices $A=ZJZ^{-1}$,  $J_-$ and $J_+$ above, except that
$n=p+q+r$ and
 $J=\diag(J_-,J_0,J_+)$ for
 a $r\times r$ matrix $J_0$
   whose all $r$ diagonal entries
have real parts 0. Then fix some $r\times r$ real diagonal matrix $D_r$,
e.g., $D_r=O_{r,r}$,
and define a {\em generalized matrix sign function} $\sign(A)$
by writing
$\sign(A)=Z\diag(-I_p,D_r\sqrt{-1},I_q)Z^{-1}$.
\end{definition}

%------------------------------------------------------------------------------

We have the following simple but basic results.

%------------------------------------------------------------------------------

\begin{theorem}\label{thsign}
Assume the generalized
matrix sign function $\sign (A)$ defined for an $n\times n$ matrix
$A=ZJZ^{-1}$.
Then for some real $r\times r$ diagonal matrix $D_r$ we have

 $$I_n-\sign (A)=Z^{-1}\diag(2I_{p},I_r-D_r\sqrt{-1},O_{q,q})Z,$$
$$I_n+\sign (A)=Z^{-1}\diag(O_{p,p},I_r+D_r\sqrt{-1},2I_{q})Z,$$
$$I_n-\sign (A)^2=Z^{-1}\diag(O_{p,p},I_r+D_r^2,O_{q,q})Z.$$
\end{theorem}

%------------------------------------------------------------------------------

\begin{corollary}\label{cosign}
Under the assumptions of Theorem \ref{thsign}
the matrix $I_n-\sign (A)^2$ has dominant eigenspace of dimension $r$
associated with the eigenvalues of the matrix $A$ that
lie on the imaginary axis $\mathcal {IA}=\{\lambda: \Re(\lambda)=0\}$,
whereas the matrices $I_n-\sign (A)$
(resp. $I_n+\sign (A)$) have dominant eigenspaces associated  with
the eigenvalues of $A$ that either lie on the left (resp. right) of the axis $\mathcal {IA}$
or  lie on this axis
and have nonzero images in $I_n-\sign (A)$
(resp. $I_n+\sign (A)$).
\end{corollary}

%------------------------------------------------------------------------------

\subsection{Ei\-gen-solv\-ing by applying matrix sign approximation
and Quad Tree construction}\label{smsres}

%------------------------------------------------------------------------------

Having the matrices $A$ and $\phi(A)=I_n-\sign (A)$
(resp. $\phi(A)=I_n+\sign (A)$)
%of Theorem \ref{thsign}
available, we can apply Algorithm \ref{fl1}
to approximate
all eigenvalues of the matrix $A$ that lie either on 
 the axis $\mathcal {IA}$ or on the left (resp. right)
from it.
The computed square matrices $L$ have dimensions
$p_+$ and $q_+$, respectively, where $p\le p_+\le p+r$
and $q\le q_+\le q+r$.
For $M=C_p$ 
this means splitting out a  degree factor of the polynomial $p(x)$
having degree $p_+$ or $q_+$. If this degree  is large,
we are likely to see dramatic
growth of the coefficients,
e.g., in the case where we split
the polynomial
$x^n+1$ into the product of two
 high degree factors, such that all roots of 
one of them have 
 positive real parts.
The problem does not arise, however, as
long as we work with matrices and 
approximate the eigenspaces. The
subdivision techniques
(cf. \cite{P00}) enable us to 
deal with
 matrices whose sizes 
are decreased recursively, and we can stop when their 
 eigenvalues are the roots of 
the small degree factors of the 
polynomial $p(x)$, and so the coefficients of these factors
are of the same order of magnitude as their roots. 
 The approach relies 
 on the following simple fact.
%$\alpha A+\sigma I$ for some fixed complex scalars
%$\alpha$ and $\sigma$ and combine the results
%with the foolowing simple fact.

%------------------------------------------------------------------------------

\begin{fact}\label{facdom}
Suppose  $\mathcal U$
and  $\mathcal V$ are two eigenspaces of $A$ and
 $\Lambda(\mathcal U)$
and  $\Lambda(\mathcal V)$
are the sets of the associated eigenvalues.
Then   
$\Lambda(\mathcal U)\cap \Lambda(\mathcal V)$ is 
the set of the  eigenvalues of $A$ 
associated  with the eigenspace $\mathcal U\cap \mathcal V$.
\end{fact}

By computing the matrix sign function of
the matrices $\alpha A-\sigma I$
for various selected pairs of complex scalars $\alpha$ and $\sigma$,
we can define the eigenspace of the matrix $A$ associated
with the eigenvalues lying in
%consequently the factor of $p(x)$ whose all roots lie in
a selected region on the complex plane
bounded by straight lines, e.g., in any 
rectangle.
In particular this supports the search policy widely known as 
{\em Quad Tree Construction}, proposed by H. Weyl in 1924 for polynomial root-finding.
 Strengthened by some modern
techniques of numerical computing, Weyl's algorithm
is practically promising and
 supports the record
Boolean complexity estimates for approximating a single root of a
 univariate polynomial \cite{P00}.
By including matrix inversions
into these computations, we define the eigenvalue
regions bounded by straight lines, their segments,
circles and their arcs.

%We discuss the respective numerical issues in the following sections.

%------------------------------------------------------------------------------

\subsection{Iterative algorithms for computing the matrix sign function
and their convergence}\label{ssignk}

%------------------------------------------------------------------------------

\cite[equations (6.17)--(6.20)]{H08} define  effective
 iterative algorithms for approximating the square root function $B^{1/2}$.
One can readily extend them to
 approximating the matrix sign function $\sign(A)=A(A^2)^{-1/2}$.
\cite[Chapter 5]{H08} presents a number of
effective iterative  algorithms devised directly for the matrix sign
function.
Among them we recall Newton's iteration
\begin{equation}\label{eqnewt}
 N_0=A,~N_{i+1}=0.5(N_i+\alpha_i~ N_i^{-1}),~i=0,1,\dots,
\end{equation}
based on the M{\"o}bius transform $x\rightarrow (x+1/x)/2$,
%and the $[l/m]$ Pad\'{e} iterations
%\begin{equation}\label{eqlm}
%$$ N_0=A,~N_{i+1}=N_i~p_{lm}(I_n-N_i^{2})(q_{lm}(I_n-N_i^{2}))^{-1},~i=0,1,\dots$$
%\end{equation}
%Here  $l$ and $m$ are fixed nonnegative integers and
%$\frac{p_{lm}(s)}{q_{lm}(s)}$ denotes the $[l,m]$ Pad\'{e}
%approximation of a polynomial or formal power series $s(x)$.
%Next we discuss convergence of iteration (\ref{eqnewt}),
%the $[1/0]$ Pad\'{e} iteration (called
%Newton--Schultz's iteration)
%\begin{equation}\label{eqnewts}
%N_0=A,~N_{i+1}=N_i(3I_n-N_i^{2})/2,~i=0,1,\dots,
%\end{equation}
and the $[2/0]$ Pad\'{e} iteration
\begin{equation}\label{eq20}
N_0=A,~N_{i+1}=(15I_n-10N_i^{2}+3N_i^{4})N_i/8,~i=0,1,\dots
\end{equation}
Theorem \ref{thsmf} implies the following simple corollary.

\begin{corollary}\label{coeigms}
Assume iterations (\ref{eqnewt}) and (\ref{eq20}) where
neither of the matrices $N_i$ is singular.
Let $\lambda=\lambda^{(0)}$ denote an eigenvalue of the matrix $N_0$
and define
\begin{equation}\label{eqnewtl}
 \lambda^{(i+1)}=(\lambda^{(i)}+(\lambda^{(i)})^{-1})/2~{\rm for}~i=0,1,\dots,
\end{equation}
%\begin{equation}\label{eqnewtsl}
%\lambda^{(i+1)}=\lambda^{(i)}(3-(\lambda^{(i)})^{2})/2,~i=0,1,\dots,
%\end{equation}
\begin{equation}\label{eq20l}
\lambda^{(i+1)}=\lambda^{(i)}(15-10(\lambda^{(i)})^{2}+3(\lambda^{(i)})^{4})/8,~i=0,1,\dots
\end{equation}
Then $\lambda^{(i)}\in \Lambda(N_i)$
  for $i=1,2,\dots$ provided
 the pairs $\{N_i,\lambda^{(i)}\}$ are defined by the pairs of equations (\ref{eqnewt}), (\ref{eqnewtl})
%,(\ref{eqnewts}) and  (\ref{eqnewtsl}),
or (\ref{eq20}), (\ref{eq20l}), respectively.
\end{corollary}

\begin{corollary}\label{coim}
In iterations (\ref{eqnewtl}) and (\ref{eq20l})
the images $\lambda^{(i)}$ of an  eigenvalue $\lambda$ of the matrix $N_0$
for all $i$
lie on the imaginary axis $\mathcal {IA}$ if so does $\lambda$.
\end{corollary}

By virtue of the following theorems, the
 sequences $\{\lambda^{(0)},\lambda^{(1)},\dots\}$
defined by equations (\ref{eqnewtl}) and (\ref{eq20l})
converge to $\pm 1$ exponentially fast
right from the start. The convergence
is quadratic for sequence (\ref{eqnewtl}) where $\Re(\lambda)\neq 0$
and cubic for  sequence (\ref{eq20l})
 where
 $|\lambda- \sign(\lambda)|\le 1/2$.

\begin{theorem}\label{thsignn} (See  \cite{H08}, \cite[page 500]{BP96}.)
Write $\lambda=\lambda^{(0)}$,
 $\delta=\sign (\lambda)$ and $\gamma=|\frac{\lambda-\delta}{\lambda+\delta}|$.
%$=|\frac{\lambda-1}{\lambda+1}|^{\delta}$.
Assume  (\ref{eqnewtl})  and  $\Re(\lambda)\ne 0$.
 Then
$|\lambda^{(i)}-\delta|\le \frac{2\gamma^{2^{i}}}{1-\gamma^{2^i}}$ for
$i=0,1,\dots$.
\end{theorem}

\begin{theorem}\label{thsign35}
Write $\delta_i=\sign (\lambda^{(i)})$ and
$\gamma_i=|\lambda^{(i)}-\delta_i|$ for $i=0,1,\dots$.
Assume  (\ref{eq20l}) and
$\gamma_0\le 1/2$.
 Then
%$\gamma_i\le \frac{4}{7}(\frac{7}{8})^{2^i}$ for
%$i=1,2,\dots$ under (\ref{eqnewtsl}), whereas
$\gamma_i\le \frac{32}{113}(\frac{113}{128})^{3^i}$ for
$i=1,2,\dots$
% under (\ref{eq20l}).
\end{theorem}

%------------------------------------------------------------------------------

\begin{proof} Complete the proof of \cite[Proposition 4.1]{BP96}
by using the bound $\gamma_0\le 0.5$.
First verify that
%$\gamma_{i+1}=\gamma_i^2|\lambda^{(i)}+2|/2$
%and therefore $\gamma_{i+1}\le \frac{7}{4}\gamma_i^2$
% for $i=0,1,\dots$ under (\ref{eqnewtsl}), whereas
$\gamma_{i+1}=\gamma_i^3|3(\lambda^{(i)})^2+9\lambda^{(i)}+8|/8$
and therefore $\gamma_{i+1}\le \frac{113}{32}\gamma_i^3$
 for
$i=0,1,\dots$.
% under (\ref{eq20l}).
Now
the claimed bounds follow by induction on $i$
because $\gamma_0\le 1/2$.
\end{proof}

%------------------------------------------------------------------------------

%\section{Real eigen-solving by means of matrix sign algorithms}\label{snumr}

%------------------------------------------------------------------------------

\subsection{Real versions of Newton's and Pad\'{e}'s iterations}\label{srnp}

%------------------------------------------------------------------------------

Having the matrix $F(A)=I_n-\sign (A)^2$
%of Theorem \ref{thsign}
available, we can apply Algorithm \ref{fl1}
to approximate
the eigenvalues of the matrix $A$ that lie on the axis $\mathcal {IA}$,
%Multiplication by $-\sqrt {-1}$ turns them
%into the real eigenvalues of the matrix $M=-A\sqrt{-1}$.
and we can devise  real
 eigen-solvers for a real $n\times n$ matrix $M$,
based on applying these techniques to the matrix
$A=M\sqrt{-1}$.
Next we modify this approach a little, to avoid involving
 nonreal values. We substitute $N_0=M$ in lieu of $N_0=A$
into matrix sign
iterations (\ref{eqnewt}) and (\ref{eq20})  and equivalently
rewrite 
them as follows,

\begin{equation}\label{eqsignr}
N_0=M,~N_{i+1}=0.5(N_i-N_i^{-1})~{\rm for}~i=0,1,\dots,
\end{equation}
%\begin{equation}\label{eqsignr3}
%N_0=M,~N_{i+1}=(N_i^3+3N_i)/2~{\rm for}~i=0,1,\dots,
%\end{equation}
\begin{equation}\label{eqsignr5}
N_0=M,~N_{i+1}=-(3N_i^5+10N_i^3+15N_i)/8~{\rm for}~i=0,1,\dots.
\end{equation}

The matrices $N_i$ 
 and the images $\lambda^{(i)}$
of every real eigenvalue $\lambda$ of $M$
are real for all $i$, whereas
the results of Theorems \ref{thsignn} and  \ref{thsign35}
are immediately extended.
The images of every nonreal point $\lambda$  converge
to the complex point $\sign(\Im(\lambda))\sqrt{-1}$ with quadratic rate
 under   (\ref{eqsignr}) if $\Re(\lambda)\neq 0$
and with cubic rate under (\ref{eqsignr5})
if $\lambda\in \mathcal D_{1/2}(\sign(\Im(\lambda))\sqrt{-1})$.
Under
the maps $M\rightarrow I_n+N_i^2$
for the matrices $N_i$ of the above iterations,
the images $1+(\lambda^{(i)})^2$
of nonreal eigenvalues
$\lambda$ of the matrix $M$ converge to 0 as long as 
the iteration is initiated 
in its basin of convergence, whereas
the images of a real point $\lambda$
are real and are at least 1 for all $i$.
Thus for sufficiently large integers $i$
we yield strong domination of
the eigenspace of the matrix $N_i$
associated with the images of the real
eigenvalues of the matrix $M$.
%and we can extend the algorithms of
%the previous sections
%keeping all computations in
%the field $\mathbb R$.
%of real numbers.

%------------------------------------------------------------------------------

\subsection{Newton's iteration with shifts for
real matrix sign function}\label{sniws}

%------------------------------------------------------------------------------

Iteration (\ref{eqsignr})
fails where for some integer $i$ the
matrix $N_i$ is singular or nearly singular,
that is has an eigenvalue equal to 0 or lying near 0,
but then we 
can
approximate this eigenvalue 
by applying
the Rayleigh Quotient Iteration
\cite[Section 8.2.3]{GL96}, \cite{BGP02/04}
or
the Inverse Orthogonal Iteration
\cite[page 339]{GL96}.

If we seek other real eigenvalues as well, we can deflate
the matrix $M$ and
apply Algorithm \ref{fl1} to the resulting matrix of a smaller size.
Alternatively we can apply it to
the matrix $ N_i+ \rho_iI_n$ for 
 a  shift $\rho_i$ randomly generated in the range
 $-r\le \rho_i \le r$ for a  positive $r$.
We choose the value $r$ reasonably small and then can expect
to
avoid degeneracy and, by virtue of Theorems 
\ref{thsignn} and \ref{thsign35}, to have 
the images of all nonreal eigenvalues of $M$
 still rapidly converging to a small neighborhood
of the points $\pm \sqrt{-1}$, thus
ensuring their isolation from the images of the real
eigenvalues.

%Alternatively we can handle the problem deterministically, by involving complex
%values into our computations. Namely, suppose that
%$\{\lambda_1,\dots,\lambda_s\}$ is the  set of all nonreal eigenvalues of $M$
%and that we know a positive scalar $\tau$ such that $0<\tau<\min_{j=1}^s|\Im(\lambda_j)|$.
%Then we can apply Newton's iteration (\ref{eqsignr}) twice, to the two matrices
%$N_{0,-}=M-\tau I_n\sqrt {-1}$ and $N_{0,+}=M+\tau I_n\sqrt {-1}$ (rather than to $N_0=M$).
%This would produce nonsingular matrices $N_{i,-}$ and  $N_{i,+}$ (replacing $N_i$) for $i=1,2,\dots$.

%Now consider the maps of an eigenvalue $\lambda$ of $M$ into
%the respective eigenvalues $\lambda^{(i)}_{\pm}$
%of the matrices $P_i=(I_n+N_{i,-})(I_n-N_{i,+})$
%and deduce from Theorem \ref{thsignn}
%that the sequence $\{\lambda^{(i)}_{\pm},i=1,2,\dots\}$
%rapidly converges as $i\rightarrow \infty$; furthermore the convergence is to 4
%if $\lambda$ is real and to 0 otherwise. For large integers $i$ this implies
% both nonsingularity and strong domination of the eigenspaces of $N_i$ associated with the images
%of the real eigenvalues of $M$.

%Usually we do not know $\min_{j=1}^s|\Im(\lambda_j)|$ but can fix a small positive value of
%the parameter $\tau$, apply the above iteration to approximate the
%eigenvalues of $M$ having imaginary parts at most $\tau$, and then select among them
% the real eigenvalues of $M$ (cf. Remark \ref{renum}).

%------------------------------------------------------------------------------

\subsection{Controlling the norms in the $[2/0]$ Pad\'{e} iterations}\label{scnpa}

%------------------------------------------------------------------------------

We have no singularity problem with
%$[1/0]$ and
iteration
 (\ref{eqsignr5}), but have numerical
problems where the norms $||N_i||$
grow large. If the nonreal eigenvalues of the matrix $N_0$
lie in the union of the two discs $\mathcal D_{1/2}(\pm\sqrt{-1})$,
then their images also stay there by virtue of a simple extension of
Theorem \ref{thsign35}, and then the norms $||N_i||$
can  be
large
only where
some real eigenvalues of
the matrices $N_i$ are absolutely large.

Now suppose
the nonreal eigenvalues of the matrix $M$ have been mapped
into the union of the two discs $\mathcal D_{y_i}(\pm\sqrt{-1})$
for $0<y_i<0.1$. (One or two steps
(\ref{eqsignr5}) move every $\mu\in\mathcal D_{1/2}(\pm\sqrt{-1})$
into the discs $\mathcal D_{y_i}(\pm\sqrt{-1})$, cf. Theorem \ref{thsign35}.)
Then the transformation $N_i\rightarrow N_i(N_i^2+2I_n)^{-1}$
confronts  excessive norm growth
by mapping all real eigenvalues of
$N_i$ into the range $[-\frac{1}{4}\sqrt {2},\frac{1}{4}\sqrt {2}]$
and mapping all nonreal eigenvalues of $N_i$ into the  discs
 $\mathcal D_{w_i}(\pm\sqrt{-1})$
for $w_i\le \frac{1+y_i}{1-2y_i-y_i^2}$.
E.g., $w_i< 0.4$  for $y_i=0.1$, whereas
 $w_i< 0.17$ for $y_i=0.05$, and then single step
 (\ref{eqsignr5}) would more than compensate for
such a minor dilation of the  discs $\mathcal D_{y_i}(\pm\sqrt{-1})$ (see Theorem \ref{thsign35}).

%------------------------------------------------------------------------------

\subsection{Moving real eigenvalues into Pad\'{e}'s
basin of convergence}\label{snum1}

%------------------------------------------------------------------------------
%------------------------------------------------------------------------------

%\subsection{Some initial observations}\label{sini}

%------------------------------------------------------------------------------

Pad\'{e}'s iteration (\ref{eqsignr5}) is 
attractive because it avoids matrix inversions
and has cubic rate of 
convergence,
but it has a quite narrow basin of convergence, 
given by the union of the discs $\mathcal D_{1/2}(\pm\sqrt{-1})$.
We can readily extend
the maps $M\rightarrow P^k$ for the matrix $P$ of (\ref{eqp}), however, 
to move all real eigenvalues of an input matrix $M$ into this basin.
Indeed for sufficiently large integers $k$ this map
moves all nonreal eigenvalues of the matrix $M$
towards the points 0 and $\infty$, while 
sending the real eigenvalues into the unit circle $\{z:~|z|=1\}$.
The maps $P^k\rightarrow 0.1~T_k\pm \sqrt {-1}~I$
for $T_k=P^k+P^{-k}$
moves this unit circle into the discs $D_{0.2}\pm \sqrt {-1}$,
both lying in the basin of convergence of 
Pad\'{e}'s iteration (\ref{eqsignr5}),
whereas this map moves the images of the nonreal eigenvalues of the 
input matrix $M$ towards $\infty$, that is keeps them outside this basin
for reasonably large integers $k$.

We can estimate the integer $k=2^{h_+}$
supporting the transforms into that basin
if we estimate the absolute values of all eigenvalues of the matrix
$P$. Towards this goal we can employ  
 Gerschg{\"o}rin discs
\cite[page 320]{GL96}, \cite[page 39]{S01}
(see also the end of the Appendix). 

%------------------------------------------------------------------------------
\section{Numerical tests}\label{sexp}
%-------------------------------------------------------------------------------

%THIS SECTION REQUIRES MORE WORK

We performed a series of numerical tests in the Graduate Center of
the City University of New York using a Dell server with a dual core 1.86 GHz
Xeon processor and 2G memory running Windows Server 2003 R2.
The test Fortran code was compiled with the GNU gfortran compiler within the Cygwin
environment.
We  generated random numbers with the random\_number
intrinsic Fortran function assuming the uniform probability distribution over the
range $\{x:~0 \leq x < 1\}$. To shift to the range $\{y:~b\le y\le a+b\}$
for fixed real $a$ and $b$, we applied the linear transform
$x\rightarrow y=ax+b$.
%We defined random complex values $x+y\sqrt{-1}$
%by random parameters $x$ and $y$ from the real line interval $[-1,1)$.

%\bigskip

%------------------------------------------------------------------------------

%{\bf Algorithms and tests}

%\bigskip

%We tested the algorithms of the previous sections applied to the matrices
%$C_p$ for polynomials $p(x)$ with random real coefficients.

We tested our algorithms  for the approximation
of the eigenvalues of $n\times n$
companion matrix $C_p$ and of the shifted matrix $C_p-sI_n$ 
defined by polynomials $p(x)$ with random real  coefficients
%DPR1 matrices (defined by random vectors ${\bf s}$ and ${\bf d}={\bf v}$ for ${\bf u}=(1)_{i=1}^n$),
for $n=64,128,256$ and by random real $s$.
For each class of matrices, each input size and
each iterative algorithm we
generated 100 input instances and run 100 tests. Our tables show
the  minimum, maximum, and average (mean) numbers of iteration loops
in these runs
(until convergence)
as well as the standard deviations in the columns marked by ``{\bf min}",
``{\bf max}", ``{\bf mean}", and ``{\bf std}", respectively.
We applied repeated squaring of Section \ref{srs}
to the matrix $C_p-sI$, where we
used 
%random real 
shifts $s$
because
polynomials $p(x)$ with random real coefficients
  tend to have all roots near the
circle $\mathcal C_1(0)$ and consequently 
  repeated squaring
of $C_p$
advances towards eigen-solving very slowly.
We applied real Newton's iteration (\ref{eqsignr}) 
to approximate the matrix sign function for the matrix $C_p$
using no
shifts. Then we applied Algorithm \ref{fl1}
to approximate real eigenvalues.

In both groups of the tests we output roots with at least four  correct decimals.
In our next group of tests we output roots with at least  three  correct decimals.
In these tests we applied real   
Pad\'{e} iteration (\ref{eqsignr5})
without stabilization
to the matrices
produced by five Newton's steps (\ref{eqsignr}).
%We also applied scaled Newton's steps (\ref{eqnewtsc})
%to the same inputs, and we show the results in Table Y.
 % AZ.
Table \ref{tablrs} displays the results of our tests of
 repeated squaring of Section \ref{srs}.
The first  three lines show the dimension of the output subspace
and the matrix $L$.
The next  three lines show the number of squarings
performed until convergence.
Table \ref{tablnw} 
%and Y  AZ
 displays the number of Newton's steps  (\ref{eqsignr})
%and scaled Newton's steps (\ref{eqnewtsc}), respectively, AZ
performed until convergence.

Table \ref{tablnwp} covers the tests where we first performed
five Newton's steps  (\ref{eqsignr}) followed by sufficiently many
Pad\'{e}  steps (\ref{eqsignr5}) required for convergence.
The first  three lines of the table show the number of the Pad\'{e} steps.
The next  three lines display the percent of the real roots of
the polynomials $p(x)$ that the algorithm computed with at least  three correct decimals
(compared to the overall number of the real eigenvalues of $L$).
The next  three lines show the increased percent of computed roots
when we refined the crude approximations
by means of
Rayleigh Quotient iteration. The iteration rapidly converged from all
these initial approximations but in many cases to the same roots from
distinct initial points.
%We also applied scaled Newton's steps (\ref{eqnewtsc})
%to the same inputs, and we show the results in Table X.
% AZ.

%Table 12.5 shows the overall numbers of Pad\'{e}  steps
%and  of the percents
%of real computed roots where we applied our algorithms to both 
%polynomials $p(x)$ and $p_{\rm rev}(x)$.

%whose all nonreal eigenvalues lied in the discs $\{x:|x+\sqrt{-1}|\le 1/2\}$
%and $\{x:|x-\sqrt{-1}|\le 1/2\}$.

%We either obtained such matrices $N_0$ by applying Algorithm \ref{fl2}
%modified to use it
%for preparing the transition to iteration (\ref{eqsignr5})
%or generated them as
%$N_0=S^{-1}\diag(\Lambda_j)_{j=1}^mS$
%for fixed real block diagonal matrices $\diag(\Lambda_j)_{j=1}^m$.

%We also tested application of repeated squaring with
%the M{\"o}bius transform
%$R(x)=\frac{1}{2}(x+1/x)$ as a means of preparation
%of the transition to
%iteration (\ref{eqsignr5})
%or as its alternative.

%In another group of our tests we replaced Stages 2
%of our Basic Algorithms with applications of the
%Power or Inverse
%Power Methods and approximated a single real eigenvalue of $M$ or its
%$r_-$ real eigenvalues for a small integer $r_-$.

%------------------------------------------------------------------------------

\begin{table}[h]
\caption{Repeated Squaring}
\label{tablrs}
\begin{center}
\begin{tabular}{|c|c|c|c|c|c|}
\hline
\textbf{$n$}&\textbf{dimension/squarings}&\textbf{min}&\textbf{max}&\textbf{mean}&\textbf{std}\\\hline
$64$   & dimension & $1$ & $10$ & $5.31$ & $2.79$ \\ \hline
$128$ & dimension & $1$ & $10$ & $3.69$ & $2.51$ \\ \hline
$256$ & dimension & $1$ & $10$ & $4.25$ & $2.67$ \\ \hline
$64$   & squarings & $6$ & $10$ & $7.33$ & $0.83$ \\ \hline
$128$ & squarings & $5$ & $10$ & $7.37$ & $1.16$ \\ \hline
$256$ & squarings & $5$ & $11$ & $7.13$ & $1.17$ \\ \hline
\end{tabular}
\end{center}
\end{table}

%------------------------------------------------------------------------------

\begin{table}[h]
\caption{Newton's iteration (\ref{eqsignr}).
}
\label{tablnw}
\begin{center}
\begin{tabular}{|c|c|c|c|c|}
\hline
\textbf{$n$}&\textbf{min}&\textbf{max}&\textbf{mean}&\textbf{std}\\\hline
$64$   & $7$ & $11$ & $8.25$ & $0.89$ \\ \hline
$128$ & $8$ & $11$ & $9.30$ & $0.98$ \\ \hline
$256$ & $9$ & $13$ & $10.22$ & $0.88$ \\ \hline

\end{tabular}
\end{center}
\end{table}

%------------------------------------------------------------------------------

\begin{table}[h]
\caption{5 N-steps (\ref{eqsignr}) + P-steps (\ref{eqsignr5})
}
\label{tablnwp0}
\begin{center}
\begin{tabular}{|c|c|c|c|c|c|}
\hline
\textbf{$n$}&\textbf{P-steps or {\%}}&\textbf{min}&\textbf{max}&\textbf{mean}&\textbf{std}\\\hline
$64$   & P-steps & $1$ & $4$ & $2.17$ & $0.67$ \\ \hline
$128$ & P-steps & $1$ & $4$ & $2.05$ & $0.63$ \\ \hline
$256$ & P-steps & $1$ & $3$ & $1.99$ & $0.58$ \\ \hline
$64$   & ${\%}$ w/o RQ steps & $0$ & $100$ & $64$ & $28$ \\ \hline
$128$ & ${\%}$  w/o RQ steps & $0$ & $100$ & $39$ & $24$ \\ \hline
$256$ & ${\%}$  w/o RQ steps & $0$ & $100$ & $35$ & $20$ \\ \hline
$64$   & ${\%}$  w/RQ steps & $0$ & $100$ & $89$ & $19$ \\ \hline
$128$ & ${\%}$  w/RQ steps & $0$ & $100$ & $74$ & $26$ \\ \hline
$256$ & ${\%}$  w/RQ steps & $0$ & $100$ & $75$ & $24$ \\ \hline
\end{tabular}
\end{center}
\end{table}

%------------------------------------------------------------------------------
\begin{table}[h]
\caption{5 N-steps (\ref{eqsignr}) + P-steps (\ref{eqsignr5})
}
\label{tablnwp}
\begin{center}
\begin{tabular}{|c|c|c|c|c|c|}
\hline
\textbf{$n$}&\textbf{P-steps or {\%}}&\textbf{min}&\textbf{max}&\textbf{mean}&\textbf{std}\\\hline
$64$   & P-steps & $2$ & $8$ & $4.26$ & $1.19$ \\ \hline
$128$ & P-steps & $2$ & $10$ & $4.20$ & $1.23$ \\ \hline
$256$ & P-steps & $2$ & $6$ & $4.24$ & $1.22$ \\ \hline
$64$   & ${\%}$ w/o RQ steps & $0$ & $100$ & $67$ & $26$ \\ \hline
$128$ & ${\%}$  w/o RQ steps & $0$ & $100$ & $43$ & $24$ \\ \hline
$256$ & ${\%}$  w/o RQ steps & $0$ & $100$ & $33$ & $23$ \\ \hline
$64$   & ${\%}$  w/RQ steps & $0$ & $100$ & $87$ & $21.3$ \\ \hline
$128$ & ${\%}$  w/RQ steps & $0$ & $100$ & $87$ & $20.5$ \\ \hline
$256$ & ${\%}$  w/RQ steps & $0$ & $100$ & $88$ & $21.5$ \\ \hline
\end{tabular}
\end{center}
\end{table}

%--- ---------------------------------------------------------------------------

%\begin{table}[h]
%\caption{Newton's Iteration =5, Repeat Squaring=1, then Pade's iteration  formula(8.9) for Pct of Real roots }
%\label{tablsge}
%\begin{center}
%\begin{tabular}{|c|c|c|c|c|c|}
%\hline
%\textbf{$n$}&\textbf{Pade Iteration or Pct of Real Roots Found}&\textbf{min}&\textbf{max}&\textbf{mean}&\textbf{std}\\\hline
%$64$   & $Pade Iteration$ & $1$ & $8$ & $0.43$ & $0.18$ \\ \hline
%$128$ & $Pade Iteration$ & $1$ & $7$ & $1.79$ & $1.09$ \\ \hline
%$256$ & $Pade Iteration$ & $1$ & $3$ & $1.42$ & $0.57$ \\ \hline
%$64$   & $Pct of Real Roots Found$ & $0$ & $1$ & $0.43$ & $0.18$\\ \hline
%$128$ & $Pct of Real Roots Found$ & $0$ & $1$ & $0.32$ & $0.18$ \\ \hline
%$256$ & $Pct of Real Roots Found$ & $0$ & $1$ & $0.31$ & $0.21$ \\ \hline
%$64$   & $Pct of Real Roots Found after RQ$ & $0$ & $1$ & $0.70$ & $0.26$ \\ \hline
%$128$ & $Pct of Real Roots Found after RQ$ & $0$ & $1$ & $0.74$ & $0.27$ \\ \hline
%$256$ & $Pct of Real Roots Found after RQ$ & $0$ & $1$ & $0.71$ & $0.25$ \\ \hline
%\end{tabular}
%\end{center}
%\end{table}

%------------------------------------------------------------------------------

\section{Conclusions
%odifications of the  Newton's matrix sign iteration by scaling
}\label{sconc}

%------------------------------------------------------------------------------
 
While presenting a number of promising approaches
we have only partly developed them to demonstrate their power
and to motivate further research efforts.
In some cases we skipped even some natural modifications.
For example, recall Newton's iteration (\ref{eqnewt}) for computing matrix sign function.
If the norms of its two terms have different orders of magnitude,
then the iteration degenerates due to rounding errors, and its convergence slows down. 
To avoid this problem we can apply scaling, that is, modify the iteration as follows,
 \begin{equation}\label{eqnewtsc}
 N_0=A,~N_{i+1}=0.5(N_i+\alpha_i~ N_i^{-1}),~\alpha_i=||N_i||/||N_i^{-1}||,~i=0,1,\dots,
\end{equation}
and similarly we can modify the variant (\ref{eqsignr}) of the iteration for real eigen-solving,

\begin{equation}\label{eqsignrsc}
N_0=M,~N_{i+1}=0.5(N_i-\alpha_i N_i^{-1})~{\rm for}~\alpha_i=||N_i||/||N_i^{-1}||~{\rm and}~i=0,1,\dots.
\end{equation}

Empirically this scaling technique substantially improves convergence,
% (see Section \ref{sexp}).
which is an example of great many potential refinements of our algorithms.
One can expect to see new advances of our approaches,
e.g., based on more intricate maps of the complex plane.
%the matrix $C_p$ and its eigenvalues.
Another potential resource of  further progress is the combination
with other matrix eigen-solvers and polynomial root-finders,
for example, a
variant of the Lanczos algorithm for real eigen-solving,
 the Rayleigh Quotient iteration,
and the
subdivision
and continued fraction methods of polynomial 
root-finding (see
 \cite{ESY06}, \cite{EMT08}, \cite{HTZ09}, \cite{MS11},
\cite{TE06}, 
\cite{YS11},  
%\cite{YVK11},
and the bibliography therein).
Various symbolic techniques  can supply auxiliary information
for our computations (e.g.,
the number of real roots and their bounds) and can
 handle the inputs
that are hard for our numerical treatment.

%\clearpage

\bigskip

% - - - - - - - - - - - - - - - - - - - - - - - - - - - - - - - - - - - - -
%\clearpage

\appendix
{\bf {\LARGE {Appendix}}}

%------------------------------------------------------------------------------

\section{Variations that involve the characteristic polynomial}\label{snum1f}

%------------------------------------------------------------------------------

In the case where $M=C_p$ is
the companion matrix of a polynomial $p(x)$,
the monic characteristic polynomial $c_{P}(x)$
for the matrix $P$ of (\ref{eqp})
equals $\gamma (x-1)^np(\frac{x+1}{x-1}\frac{\sqrt {-1}}{a})=
\gamma (x-1)^np(1-\frac{2}{x-1}\frac{\sqrt {-1}}{a})$
for a scalar $\gamma$.
We can obtain its coefficients by performing two shifts of the variable 
(see \cite[Chapter 2]{P01} on this operation) and
the single reversion of the polynomial coefficients.
When this is done we  can replace $k$ repeated squarings of the matrix $P$
with
$k$ steps of the Dandelin's 
root-squar\-ing iteration, 
also attributed to some later works by Lobachevsky and Gr\"{a}ffe (see \cite{H59}),
\begin{equation}\label{eqgr}
p_{i+1}(x)=(-1)^n p_i(\sqrt x) p_i(\sqrt {-x}),~~i=0,1,\dots,k-1
\end{equation}
for $p_0(x)=c_{P}(x)$.
We have $p_i(x)=\prod_{j=1}^n(x-\lambda_j^{2^i})$, so that the $i$th iteration step
squares the roots of the polynomial $p_{i-1}(x)$ for every $i$.
Every root-squaring step (\ref{eqgr}) essentially amounts to
polynomial multiplication and can be performed in $O(n\log n)$
flops. One can
improve numerical stability
by applying 
modifications in  \cite{MZ01}, which use  
 order of $n^2$ flops per iteration.
Having computed the polynomial $p_k(x)$ for a sufficiently large integer $k$,
we have its roots on the unit circle sufficiently well isolated from its other roots.
The application of the algorithm of Section \ref{scnrs} to the matrix  $C_{p_k}$,
the companion matrix of this polynomial,
yields its roots lying on the circle  $\mathcal C_1$
(they are the eigenvalues of  the matrix $C_{p_k}$).
From these roots we can recover the roots $\mu$ of the polynomial $c_{P}(x)=p_0(x)$
by means of the descending techniques of \cite{P95}
(applied also in \cite{P96}, \cite{P97}, \cite{P01/02}, and \cite[Stage 8 of Algorithm 9.1]{PZ10/11}),
and then can recover the real roots $\lambda$ of
the polynomial $p(x)$
 from the values $\mu$ by
 applying the expression in part (a) of Fact \ref{facrcsqr}.
In this approach we can readily approximate 
 the eigenvalues of the matrix $P$ 
from the origin as the root radii of the
characteristic 
polynomial $c_P(x)=\det(xI_n-P)$.
Indeed as long as we are given the coefficients we can 
 approximate all the root radii 
with relative errors of at most $1$\%
 by using 
$O(n\log n)$ flops (see  
 \cite{B96}, \cite{BF00}, \cite{P00}, \cite{P01/02},
 \cite{S82}).
%------------------------------------------------------------------------------

\begin{remark}\label{refctr}
Having isolated the roots of $p_k(x)$ on the circle $\mathcal C_1$
from its other roots, we can apply the algorithms of  \cite{K98},
 \cite{P95}, \cite{P96}, \cite{P01/02}, \cite{S82} to split out the factor
$f(x)$  sharing with the polynomial 
precisely all the roots that lie on the  circle $\mathcal C_1$.
Then these roots can be moved into the real line and 
then readily approximated based on the Laguerre
or modified  Laguerre algorithms
 \cite{P64}, \cite{HPR77}, \cite{DJLZ96}, \cite{DJLZ97}, and \cite{Z99}.
Numerical problems can be caused by
potentially dramatic growth
of the coefficients of
the polynomial $p_k(x)$ in the transition to the factor $f(x)$
unless its degree is small.
\end{remark}

%------------------------------------------------------------------------------

{\bf Acknowledgements:}
Our research has been supported by NSF Grant CCF--1116736 and
PSC CUNY Awards 64512--0042 and 65792--0043.

%------------------------------------------------------------------------------

%------------------------------------------------------------------------------

\end{document}